\documentclass[11pt]{amsart}

\usepackage{amsmath}
\usepackage{amstext}
\usepackage{amssymb}
\usepackage{amsthm}
\usepackage{enumerate}

\makeatletter
\def\imod#1{\allowbreak\mkern10mu({\operator@font mod}\,\,#1)}
\makeatother

\newtheorem{theorem}{Theorem}[section]

\newtheorem{prop}[theorem]{Proposition}
\newtheorem{lemma}[theorem]{Lemma}

\newtheorem{corollary}[theorem]{Corollary}

\theoremstyle{definition}
\newtheorem{definition}[theorem]{Definition}

\theoremstyle{remark}
\newtheorem{remark}[theorem]{Remark}

\theoremstyle{remark}

\numberwithin{equation}{section}

     \DeclareMathOperator{\Aut}{Aut}
    \DeclareMathOperator{\lh}{lh}

    \DeclareMathOperator{\supp}{supp}
    
    \DeclareMathOperator{\ran}{ran}
    
    \DeclareMathOperator{\proj}{proj}
    \DeclareMathOperator{\Rep}{Rep}

     \DeclareMathOperator{\II}{II}
    \DeclareMathOperator{\vN}{vN}
    \DeclareMathOperator{\III}{III}
    \DeclareMathOperator{\I}{I}

    \DeclareMathOperator{\rep}{Rep}

    \newcommand{\actson}{\curvearrowright}

\def\R{{\mathbb R}}
\def\C{{\mathbb C}}

\def\N{{\mathbb N}}

\def\Q{{\mathbb Q}}

\begin{document}

\title[The conjugacy relation on unitary representations]{The conjugacy relation on unitary representations}

\author{Greg Hjorth{$^{\alpha,\beta}$}}

\author{Asger T\"ornquist$^{\gamma}$}

\thanks{$^\alpha$ Note from the second author: Greg Hjorth, who was my dear friend, colleague, and Ph.D. thesis advisor (at UCLA, 2001-2005), died unexpectedly on January 13, 2011, while we were finishing work on this paper.}

\thanks{$^\beta$ Greg Hjorth was supported by an Australian Professorial Fellowship from the Australian Research Council.}

\thanks{$^\gamma$ Asger T\"ornquist was supported in part by the Austrian Science Fund grant no. P 19375-N18, and a Marie Curie fellowship, grant no. IRG-249167, from the European Union.}



\begin{abstract}
In a 1965 paper, E.G. Effros asked the question if the conjugacy relation for unitary representations of a locally compact second countable group is a Borel equivalence relation. In this paper we answer this question affirmatively. This also settles a recent question raised by A.S. Kechris regarding the complexity of unitary equivalence of probability measure preserving actions of countable discrete groups.
\end{abstract}

\maketitle

\section{Introduction}

Let $H$ be an infinite dimensional separable complex Hilbert space with inner product denoted by $\langle\cdot,\cdot\rangle$ and norm $\|\cdot\|$. As usual, let $\mathcal B(H)$ denote the set of bounded operators on $H$. Let $U(H)$ denote the unitary group of $H$, which we equip with the strong (equivalently, weak) operator topology\footnote{We refer to \cite{pedersen89} for basic definitions.}. Recall that $U(H)$ is a Polish group when given this topology.

Let $\Gamma$ be a discrete countable group. A unitary representation of $\Gamma$ is a homomorphism $\pi:\Gamma\to U(H)$. Define
$$
\Rep(\Gamma,H)=\{\pi:\Gamma\to U(H): (\forall\gamma_1,\gamma_2\in\Gamma): \pi(\gamma_1\gamma_2)=\pi(\gamma_1)\pi(\gamma_2)\},
$$
the \emph{space of unitary representations of $\Gamma$ on $H$}, which we equip with the subspace topology inherited from $U(H)^\Gamma$, when the latter is given the product topology. Since $\Rep(\Gamma,H)$ is a closed subset of $U(H)^\Gamma$, $\Rep(\Gamma,H)$ is a Polish space. For notational ease, we usually write $\pi_\gamma$ for $\pi(\gamma)$, when $\pi\in\Rep(\Gamma,H)$.

We say that two representations $\sigma,\pi\in\Rep(\Gamma,H)$ are \emph{conjugate} if there is $U\in U(H)$ such that $U\sigma_\gamma U^{-1}=\pi_\gamma$ for all $\gamma\in\Gamma$, and we write $\sigma\simeq\pi$ if this is the case. It is clear that the conjugacy relation $\simeq$ is induced by the natural conjugation action of $U(H)$ on $\Rep(\Gamma,H)$, and so $\simeq$ is an analytic equivalence relation, i.e., an analytic subset of $\Rep(\Gamma,H)^2$.

The question was raised by Effros in \cite[pp. 1163--1164]{effros65} if $\simeq$ is in fact Borel. The purpose of this paper is to prove that this is indeed the case:

\begin{theorem}\label{mainthm1}
The conjugacy relation in $\rep(\Gamma, H)$ is $F_{\sigma\delta}$.
\end{theorem}

This also answers a question raised by Kechris in \cite[p. 123, (IIIb) and (IVb)]{kechris10}, who asked if unitary equivalence of probability measure preserving actions of $\Gamma$ is Borel. We note that Theorem \ref{mainthm1} stands in contrast to the recent result of Foreman, Rudolph and Weiss \cite{foruwe}, who have shown that the conjugacy relation on ergodic transformations in the group of measure preserving transformations $\Aut(X,\mu)$ is complete analytic, and so cannot be Borel, when $(X,\mu)$ is a standard non-atomic probability space.

Theorem \ref{mainthm1} applies more generally to representations of \emph{separable involutive Banach algebras} in the sense of \cite[Appendix B]{nielsen80}, and so in particular it applies to separable C$^*$-algebras and to unitary representations of second countable locally compact groups. If $\Gamma$ is a separable involutive Banach algebra, then we let $\rep(\Gamma,H)$ denote the set of $*$-homomorphisms from $\Gamma$ to $\mathcal B(H)$. We give $\rep(\Gamma,H)$ the initial topology induced by the maps
$$
\sigma\mapsto \sigma_\gamma(\xi),
$$
ranging over all $\gamma\in \Gamma$ and $\xi\in H$. It can be seen that $\rep(\Gamma,H)$ is Polish in this topology. The (unitary) conjugacy relation $\simeq$ in $\rep(\Gamma,H)$ is defined as in the case of countable discrete groups, and we have:

\begin{theorem}\label{mainthm2}
Let $\Gamma$ be a separable involutive Banach algebra. Then the conjugacy relation in $\rep(\Gamma, H)$ is $F_{\sigma\delta}$.
\end{theorem}

Theorems \ref{mainthm1} and \ref{mainthm2} have identical proofs, and below we will use $\Gamma$ to stand for either a countable discrete group or a separable involutive Banach algebra.

The paper is organized as follows: First, in \S 2, we will recall various basic notions from the theory of von Neumann algebras and the direct integral theory of unitary representations, and establish a simple lemma that will facilitate hierarchy complexity calculations later. The proof of Theorems \ref{mainthm1} and \ref{mainthm2} is given in \S 3, where we also show that the result is optimal, and discuss some consequences related to Borel reducibility of equivalence relations.

\medskip

We would like to thank Roman Sasyk for his careful reading of an earlier version for this paper, and for his useful comments and suggestions. We would also like to thank Simon Thomas for pointing out Corollary \ref{c.simon} below.

\section{Preliminaries}

In what follows, $H$ will always denote an infinite dimensional separable complex Hilbert space equipped with the topology induced by its norm. We let $H^1=\{\xi\in H:\|\xi\|\leq 1\}$ be the unit ball of $H$ and $\mathcal B^1(H)=\{T\in\mathcal B(H):\|T\|\leq 1\}$, where $\|T\|$ is the operator norm of $T$.


\subsection{Elementary relations and functions} The following Lemma collects the basic complexity calculations that we need for our arguments.


\begin{lemma}\label{l.transfer}
(a) The function 
$$
\mathcal B(H)\times H\times H\to\C:(T,\xi,\eta)\mapsto\langle T\xi,\eta\rangle
$$
is continuous when $\mathcal B(H)$ has the weak topology.

(b) The function $\mathcal B(H)\times H\to H: (T,\xi)\to T\xi$ is continuous when $\mathcal B(H)$ is given the strong topology (but not the weak.)

(c) The relations $Q^{<},Q^{\leq}\subseteq H\times H\times\C\times\R_+\times\mathcal B(H)$ defined by
$$
Q^{<}(\xi,\eta,z,\varepsilon,T)\iff |\langle T\xi,T\eta\rangle -z|< \varepsilon
$$
$$
Q^{\leq}(\xi,\eta,z,\varepsilon,T)\iff |\langle T\xi,T\eta\rangle -z|\leq \varepsilon
$$
are $F_\sigma$ when $\mathcal B(H)$ is given the weak topology.

(d) The basic strongly closed sets $\{T\in\mathcal B(H): \|T\xi-\eta\|\leq\varepsilon\}$ are also weakly closed.

(e) The basic strongly open sets $\{T\in\mathcal B(H): \|T\xi-\eta\|<\varepsilon\}$ are weakly $F_\sigma$, and so the strongly open sets are weakly $F_\sigma$.

(f) The map $H\times H\times\mathcal B(H)\times \mathcal B(H)\to\C: (\xi,\eta,T,S)\mapsto \langle T\xi,S\eta\rangle$ is continuous when the first $\mathcal B(H)$ has the weak topology and the second has the strong topology (or when the opposite is the case.)
\end{lemma}


\begin{proof}
Verification of (a) and (b) are routine and left to the reader.

(c) Fix an orthonormal basis $(e_i)$ for $H$. Then
$$
|\langle T\xi,T\eta\rangle-z|<\varepsilon \iff (\exists N)(\exists q\in\Q_+)(\forall k\geq N)|\sum_{l=1}^k\langle T\xi,e_l\rangle\langle T\eta, e_l\rangle^*-z|\leq\varepsilon-q.
$$
Since the coefficients in the sum are weakly continuous, this gives a $F_\sigma$ definition of $Q^<$. That $Q^{\leq}$ is $F_\sigma$ now also follows.

(d) Since $\|\xi\|=\sup\{|\langle \xi,\eta\rangle|: \|\eta\|\leq 1\}$ we have
$$
\|T\xi-\eta\|\leq\varepsilon\iff (\forall\zeta\in H^1)|\langle T\xi-\eta,\zeta\rangle|\leq\varepsilon,
$$
which gives a weakly closed definition of the basic strongly closed ball.

(e) follows directly from (d), and (f) follows from (a) and (b).
\end{proof}


\begin{definition}
The relation $R\subseteq\rep(\Gamma,H)\times\rep(\Gamma,H)\times\mathcal B(H)$ is defined by
$$
R(\sigma,\pi,T)\iff (\forall \gamma\in\Gamma) T\sigma_\gamma=\pi_\gamma T.
$$
The set $R_{\sigma,\pi}=\{T\in\mathcal B(H):R(\sigma,\pi,T)\}$ is the set of \emph{intertwiners} of $\sigma$ and $\pi$. We let $R_\sigma=R_{\sigma,\sigma}$, and
$$
R^1=R\cap \rep(\Gamma,H)\times\rep(\Gamma,H)\times\mathcal B^1(H).
$$
Finally, let $R^{\rm u}_{\sigma,\pi}$ denote the set of partial isometries in $R^1_{\sigma,\pi}$, i.e. $u\in R^1_{\sigma,\pi}$ such that $u^*u$ and $uu^*$ are (orthogonal) projections.

\end{definition}


\begin{lemma}\label{l.relation_R}
The relations $R$ and $R^1$ are closed when $\mathcal B(H)$ has the weak topology. Moreover, if $T\in R_{\sigma,\pi}$ then $T^*\in R_{\pi,\sigma}$. In particular, $R_\sigma$ is a von Neumann algebra acting on $H$.
\end{lemma}


\begin{proof}
We claim that
\begin{equation}\label{eq.R}
R(\sigma,\pi,T)\iff (\forall\gamma\in\Gamma)(\forall\varepsilon>0)(\forall \xi,\eta\in H) |\langle\sigma_\gamma(\xi),T^*\eta\rangle-\langle T\xi,\pi_{\gamma^{-1}}(\eta)\rangle|\leq\varepsilon.
\end{equation}
Note that this equivalence gives a weakly closed definition of $R$ by (f) of Lemma \ref{l.transfer} and the fact that $T\mapsto T^*$ is weakly continuous. To see the equivalence, just note that it follows from the right hand side of \eqref{eq.R} that
$$
(\forall\varepsilon>0)(\forall \xi,\eta\in H) |\langle T\sigma_g(\xi),\eta\rangle-\langle \pi_g T\xi,\eta\rangle|\leq\varepsilon,
$$
so $\|T\sigma_g\xi-\pi_g T\xi\|=0$ for all $\xi\in H$, and so $T\sigma_g=\pi_g T$. Finally, \eqref{eq.R} gives $R(\sigma,\pi,T)\iff R(\pi,\sigma,T^*)$.
\end{proof}


\subsection{Some von Neumann algebra notions} In this section we recall the basic notions from the theory of von Neumann algebras that are needed for the proof of Theorems \ref{mainthm1} and \ref{mainthm2}.

The set of (orthogonal) projections in $\mathcal B(H)$ will be denoted $P(H)$, and for $M$ a von Neumann algebra (e.g., $M=R_\sigma$), we let $P(M)$ be the set of projections in $M$. Also, $M^+$ denotes the set of positive operators in the von Neumann algebra $M$, and, as usual, for $x,y\in M$ we write $x\leq y$ when $y-x\in M^+$.

If $\sigma\in\Rep(\Gamma,H)$ and $p\in P(R_\sigma)$, then the range of $p$ is a $\sigma$-invariant closed subspace, and we denote by $\sigma|p$ the restriction of $\sigma$ to $\ran(p)$. If $\sigma,\pi\in\Rep(\Gamma,H)$ and $p\in P(R_\sigma)$, $q\in P(R_\pi)$, then we will say that $p$ is \emph{equivalent} to $q$, written $p\sim q$, if $\sigma|p\simeq \pi|q$. This is clearly the same as saying that there is a partial isometry $u\in R_{\sigma,\pi}$ such that $u^*u=p$ and $uu^*=q$. In particular, $p,q\in P(R_\sigma)$ are equivalent precisely when they are Murray-von Neumann equivalent in $R_\sigma$, see \cite{blackadar06}. Also, for $p\in R_\sigma$ and $q\in R_\pi$, write $p\precsim q$ if $p$ is equivalent to a sub-projection of $q$. Further, we  will write $\sigma\precsim\pi$ if $1\in R_\sigma$ is equivalent to a projection in $R_\pi$, and so in particular, we have $\sigma|p\precsim \pi|q$ if and only if $p\precsim q$. Throughout the paper we will use the basic (Murray-von Neumann) comparison theory of projections freely, and refer to \cite{blackadar06} for background.

When $\Gamma$ is group, we will say that the representations $\sigma$ and $\pi$ are \emph{disjoint}, written $\sigma\perp\pi$, if $R_{\sigma,\pi}=\{0\}$. When $\Gamma$ is an involutive Banach algebra, we define $\sigma\perp\pi$ to hold if either $\sigma$ or $\pi$ is a trivial representation (i.e., $\sigma_a=0$ for all $a\in\Gamma$), or $R_{\sigma,\pi}=\{0\}$.\footnote{Our definition of $\perp$ is made so that Theorem \ref{t.directintdecomp} can be stated without making the additional assumption that the representations are non-degenerate.} Effros has shown that $\perp$ is a Borel relation, see \cite{effros65}.

\begin{definition}
A representation $\sigma\in\rep(\Gamma,H)$ is called a \emph{factor representation} (or a \emph{primary representation}), if $R_\sigma$ is a factor, i.e., if the centre of $R_\sigma$ consists of multiples of $1$.
\end{definition}


\begin{remark}\label{rem.transitivity}
It is useful to note that for primary representations, the relation $\not\perp$ is transitive: If $\sigma,\pi,\rho\in \Rep(\Gamma,H)$ are primary and $\sigma\not\perp\pi$ and $\pi\not\perp\rho$, then $\sigma\not\perp\rho$. This follows since in a factor $M$, it holds for all $p,q\in P(M)$ that either $p\precsim q$ or $q\precsim p$. Since $\not\perp$ is also reflexive and symmetric, it is an equivalence relation on the set of factor representations. 
\end{remark}

The following is a well-known consequence of the polar decomposition theorem (see e.g. \cite[1.12.1]{sakai71}):


\begin{lemma}\label{l.polardecomp}
Let $\sigma,\pi\in\rep(\Gamma,H)$ and suppose $T\in R_{\sigma,\pi}$. Let $T=u|T|$ be the polar decomposition. Then $|T|\in R_\sigma$ and $u\in R_{\sigma,\pi}$. In particular, if $p_0\in P(R_\sigma)$, $q_0\in P(R_\pi)$ and $\sigma|p_0\not\perp\pi|q_0$ then there is a subprojection of $p_0$ which is equivalent to a subprojection of $q_0$.
\end{lemma}


\begin{proof}
Let $T\in R_{\sigma,\pi}$. Then from Lemma \ref{l.relation_R} we have $TT^*\in R_\sigma$, and since $R_\sigma$ is von Neumann algebras we have $|T|=(T^*T)^{\frac 1 2}\in R_\sigma$. Moreover, $p=u^*u\in P(R_\sigma)$, since it is the projection onto $\ker(|T|)^\perp$. Since we have
$$
u\sigma_\gamma(|T|\xi)=T\sigma_\gamma(\xi)=\pi_\gamma T(\xi)=\pi_\gamma(u|T|\xi),
$$
and since $\ran(|T|)$ is dense in $\ran (p)$ (see \cite[1.12.1]{sakai71}), it follows that 
$$
u\sigma_\gamma(\xi)=up\sigma_\gamma(\xi)=u\sigma_\gamma(p\xi)=\pi_\gamma u(p\xi)=\pi_\gamma u(\xi)
$$ 
for all $\gamma\in\Gamma$ and $\xi\in H$. Thus $u\in R_{\sigma,\pi}$, as required.
\end{proof}


A primary representation $\sigma$ will be called \emph{type} I, II, or III according to if $R_\sigma$ is a factor of type I, II or III, see \cite{blackadar06}. Following further standard terminology, we will say that a primary representation is \emph{finite} if the identity is a finite projection, \emph{semifinite} if it is type I or II, and \emph{purely infinite} if it is type III. We shall also need the following fundamental fact (see e.g. \cite[III.2.5.8]{blackadar06}):


\begin{theorem}
Every finite von Neumann factor $M$ acting on a separable Hilbert space admits a faithful ultraweakly continuous trace $\tau:M\to \C$ such that $\tau(1)=1$.
\end{theorem}


\subsection{Direct integral theory} We now review the direct integral theory of unitary representations. The main reference for this is \cite{nielsen80}.


\begin{definition}
Let $\sigma\in\rep(\Gamma,H)$. By a \emph{direct integral decomposition of $\sigma$ into primary representations} we mean a pair $((X,\mu),\sigma_x\actson H_x)$, where $(X,\mu)$ is a standard measure space and $\sigma_x\in\Rep(\Gamma, H_x)$ is a measurable assignment of representations $\sigma_x$ acting on Hilbert spaces $H_x$, and the following holds:

(a) $\sigma_x$ is a primary representation for all $x\in X$;

(b) $\sigma$ is isomorphic to 
$$
\int\sigma_xd\mu(x);
$$

(c) $\mu(\{x: (\forall y\in X) y\neq x\implies \sigma_x\perp\sigma_y\})=1$.
\end{definition}

\noindent Without any real impact, we could of course replace (c) with

\medskip

(c$^\prime$) For all $x,y\in X$, if $x\neq y$ then $\sigma_x\perp\sigma_y$.

\medskip

Using the direct integral decomposition of $R_\sigma$ into factors, one obtains (see \cite[III.5.1.14]{blackadar06} and \cite[Ch. 3]{nielsen80}):


\begin{theorem}\label{t.directintdecomp}
Every $\sigma\in\Rep(\Gamma, H)$ admits a direct integral decomposition $((X,\mu), \sigma_x\actson H_x)$ into primary representations. Moreover, this decomposition is essentially unique in the following sense: If $((Y,\nu),\sigma_y\actson H_y)$ is any other direct integral decomposition of $\sigma$, then there is a measure-class preserving bijection $\phi:X\to Y$ such that $\sigma_x\simeq \sigma_{\phi(x)}$ for almost all $x\in X$.
\end{theorem}



Not only do representations admit direct integral decompositions, but so do intertwiners. If $\sigma\in\rep(\Gamma,H)$ and $((X,\mu),\sigma_x\actson H_x)$ is a decomposition into primary representations, then any $T\in R_\sigma$ can be written as $T=\int T_xd\mu(x)$ essentially uniquely, where $T_x\in R_{\sigma_x}$. We define the \emph{central support} of $T$ to be
$$
\supp(T)=\{x\in X:T_x\neq 0\}.
$$
This is well-defined modulo a $\mu$-null set. From \cite[Theorem 12.4]{nielsen80} we further have:


\begin{lemma}\label{l.decomp}
Let $\sigma_i\in\rep(\Gamma,H)$, $i=0,1$, and let $((X_i,\mu_i), \sigma_{i,x}\actson H_{i,x})$ be the decomposition of $\sigma_i$ into factor representations. Suppose $u\in R^{\rm u}_{\sigma_0,\sigma_1}$. Then there is a measure class preserving bijection $\phi:\supp(u^*u)\to \supp(uu^*)$ and isometries $u_x:H_{0,x}\to H_{1,\phi(x)}$ such that $u_x\in R^{\rm u}_{\sigma_{0,x},\sigma_{1,\phi(x)}}$ and
$$
u=\int \theta(x)^{\frac 1 2} u_xd\mu_0(x),
$$
i.e., $(u\xi)(\psi(x))=\theta(x)^{\frac 1 2}u_x(\xi(x))$, and where $\theta=\frac{d\mu_1}{d\psi[\mu_0]}$.

\end{lemma}




\section{The proof}

We will now give the proof of Theorem \ref{mainthm1} and \ref{mainthm2}. First we define a certain $F_{\sigma \delta}$ relation $S_0\in\rep(\Gamma,H)^2$, along with the auxiliary relations $S_1$ and $S_2$. The set of finite sequences in $H$ is denoted by $H^{<\N}$.


\begin{definition}
The relation $S_2\subseteq \rep(\Gamma,H)\times\rep(\Gamma, H)\times \Q_+\times H^{<\N}\times\mathcal B_1(H)$ is defined as
$$
S_2(\sigma,\pi,\varepsilon, \vec\xi, T)\iff R^1(\sigma,\pi,T)\wedge (\forall i,j<\lh(\xi)) |\langle\vec\xi_i,\vec\xi_j\rangle-\langle T(\vec\xi_i),T(\vec\xi_j)\rangle|<\varepsilon,
$$
and the relation $S_1\subseteq\rep(\Gamma,H)^2\times\Q_+\times H^{<\N}$ is defined as
$$
S_1(\sigma,\pi,\varepsilon,\vec\xi)\iff (\exists T) S_2(\sigma,\pi,\varepsilon,\vec\xi,T).
$$
Finally, the relation $S_0\subseteq \rep(\Gamma,H)^2$ is defined as
$$
S_0(\sigma,\pi)\iff (\forall\varepsilon\in\Q_+)(\forall \vec\xi) S_1(\sigma,\pi,\varepsilon,\vec\xi)\iff (\forall\varepsilon\in\Q_+)(\forall \vec\xi)(\exists T) S_2(\sigma,\pi,\varepsilon,\vec\xi,T).
$$
\end{definition}


\begin{lemma}\label{l.relations}
Let $\Gamma$ and $H$ be as above, and give $\mathcal B^1(H)$ the \emph{weak} topology. Then

(i) The relation $S_2$ is $F_{\sigma}$ (i.e., $\mathbf{\Sigma}^0_2$.)

(ii) The relation $S_1$ is $F_\sigma$.

(iii) The relation $S_0$ is $F_{\sigma\delta}$ (i.e., $\mathbf{\Pi}^0_3$.)
\end{lemma}

We need the following easy topological fact, the proof of which is left to the reader.


\begin{lemma}\label{l.topological}
Let $X$ be a Polish space, $Y$ a compact Polish space, and $A\subseteq X\times Y$. Then:

(1) If $A$ is closed then $\proj_X(A)=\{x\in X:(\exists y\in Y) (x,y)\in A\}$ is closed.

(2) If $A$ is $F_\sigma$ then $\proj_X(A)$ is $F_\sigma$.
\end{lemma}

\begin{proof}[Proof of Lemma \ref{l.relations}]

(i) Follows directly from Lemma \ref{l.transfer} (c).

(ii) Directly from Lemma \ref{l.topological} (2).

(iii) Clear, since it suffices to quantify over sequences in a countable dense subset of $H$.
\end{proof}

We will prove Theorem \ref{mainthm1} and \ref{mainthm2} by proving


\begin{theorem}\label{mainthmv2}
Let $\sigma,\pi\in\rep(\Gamma,H)$. Then $\sigma\simeq\pi$ if and only if $S_0(\sigma,\pi)$ and $S_0(\pi,\sigma)$.
\end{theorem}

We first prove Theorem \ref{mainthmv2} for $\sigma$ and $\pi$ primary. Note that when $\sigma,\pi$ are primary it follows by an easy application of Lemma \ref{l.polardecomp} that if $S_0(\sigma,\pi)$ holds and $\sigma$ is type $\III$, then so is $\pi$, and $\sigma\simeq\pi$, since all non-zero projections in a type III factor are equivalent. This leaves us to deal with the semifinite case.


\begin{lemma}\label{l.unitarywitness}
Let $T\in R^1_{\sigma,\pi}$ and let $T=u|T|$ be the polar decomposition, and let $\vec{\xi}$ be a sequence of vectors in $H^1$. Let $p=u^*u$. Then for $0<\varepsilon<1$:

(1) If $S_2(\sigma,\pi,\varepsilon^2,\vec{\xi},T)$ then $\|\vec{\xi}_i-p\vec{\xi_i}\|<\varepsilon$.

(2) If $\|p\vec{\xi}_i-\vec{\xi}_i\|<\varepsilon$ then $S_2(\sigma,\pi,\varepsilon,\vec{\xi}, u)$ holds. 

\noindent In particular, the following are equivalent:

(I) $S_0(\sigma,\pi)$.

(II) For all $\varepsilon>0$ and $\vec{\xi}\in H^{<\N}$ there is a partial isometry $u\in R^1_{\sigma,\pi}$ such that $\|u^*u\vec{\xi_i}-\vec{\xi}_i\|<\varepsilon$.

(III) For some (any) orthonormal basis $(e_n)$ for $H$ it holds that for all $\varepsilon>0$ and $n\in\N$ there is a partial unitary $u\in R^1_{\sigma,\pi}$ such that $\|u^*u e_i-e_i\|<\varepsilon$, for all $i\leq n$.
\end{lemma}


\begin{proof}
(1) Since $|T|\leq 1$, $T=u|T|=up|T|p$ and $u$ is a partial unitary, it follows from
$$
|\langle\vec{\xi}_i,\vec{\xi}_j\rangle-\langle T\vec{\xi}_i,T\vec{\xi}_j\rangle|=|\langle\vec{\xi}_i,\vec{\xi}_j\rangle-\langle |T|p\vec{\xi}_i,|T|p\vec{\xi}_j\rangle|<\varepsilon^2
$$
that $1-\|p\vec{\xi}_i\|^2<\varepsilon^2$, from which $\|\vec{\xi}-p\vec{\xi}_i\|<\varepsilon$ follows, since $p$ is an orthogonal projection.

(2) We have
$$
|\langle\vec{\xi}_i,\vec{\xi}_j\rangle-\langle u\vec{\xi}_i,u\vec{\xi}_j\rangle|=|\langle\vec{\xi}_i,\vec{\xi}_j\rangle-\langle u^*u\vec{\xi}_i,\vec{\xi}_j\rangle|=|\langle\vec{\xi}_i-p\vec{\xi}_i,\vec{\xi}_j\rangle|\leq \|\vec{\xi}_i-p\vec{\xi}_i\|<\varepsilon.
$$
Finally, the equivalence (I) and (II) are clear from (1) and (2), and the easy proof that (II) is equivalent to (III) is left to the reader.
\end{proof}


Before proceeding, we make the following easy observation about rearrangement of projections that we use below several times: Suppose $M$ is a semifinite factor, and that $p_i\in M$ are finite projections such that $p_i\precsim p_{i+1}$ for all $i\in\N$. Then we can find projections $q_i\in M$ such that $q_i\sim p_i$ and $q_i\leq q_{i+1}$ for all $i\in\N$. To prove this, let $q_1=p_1$ and let $u\in M$ be a partial unitary such that $p_1=u^*u$ and $p_1'=uu^*\precsim p_2$. Since we must have $1-p_1\sim 1-p_1'$, choose $v\in M$ witnessing this. Since $u$ and $v$ have orthogonal domain and range projections, $w=u+v\in M$ is unitary, and an easy calculation shows that $wp_1w^*=p_1'$. Let $q_2=w^* p_2 w$, and proceed inductively.


\begin{lemma}
Let $\sigma,\pi\in\rep(\Gamma,H)$ be primary representations. Then $S_0(\sigma,\pi)$ implies $\sigma\precsim\pi$.
\end{lemma}


\begin{proof}
As observed after the statement of Theorem \ref{mainthmv2}, it suffices to consider the case when $R_\sigma$ is semifinite.

Assume first that $R_\sigma$ is a finite factor, and let $\tau:R_\sigma\to\C$ be the normalized trace. Since $S_0(\sigma,\pi)$ holds, we can find partial isometries $u_i\in R_{\sigma,\pi}$ such that $p_i=u_i^*u_i\to 1$ strongly, and so $\tau(p_i)\to 1$. Thus, after possibly going to a subsequence, we can assume that $p_i\precsim p_{i+1}$ for all $i\in\N$. By the observation preceding the Lemma, we can then find projections $q_i\in R_\sigma$ such that $q_i\sim p_i$ and $q_i\leq q_{i+1}$ for all $i\in\N$. Note that if $u\in R_\sigma^{\rm u}$ is such that $q_i=u^*u$ and $p_i=uu^*$, then $u_iu\in R_{\sigma,\pi}^{\rm u}$ witnesses the equivalence of $q_i$ and $\bar p_i=u_iu_i^*\in R_\pi$.

Clearly $\bar p_i\precsim \bar p_{i+1}$, and so we can again find $\bar q_i\in R_\pi$ such that $\bar p_i\sim \bar q_i$ and $\bar q_i\leq \bar q_{i+1}$ for all $i\in\N$.

Now let $r_1=q_1$, $\bar r_1=\bar q_1$, and for $i>1$, let $r_i=q_i-q_{i-1}$, $\bar r_i=\bar q_i-\bar q_{i-1}$. Then $(r_i)$ and $(\bar r_i)$ are sequences of mutually orthogonal projections, and since $\sum_{i=1}^n r_i=q_n$, we have $\sum_{i=1}^\infty r_i=1$.  Furthermore, $r_i\sim \bar r_i$ for all $i\in\N$ (see \cite[III.1.3.8]{blackadar06}), so fix $w_i\in R_{\sigma,\pi}^{\rm u}$ such that $r_i=w_i^*w_i$ and $\bar r_i=w_iw_i^*$. Then $\sum_{i=1}^\infty w_i\in R_{\sigma,\pi}$ is an isometry, and so $\sigma\precsim\pi$.

Dropping the assumption that $R_\sigma$ is finite, then it follows from the above that $\sigma|p\precsim\pi$ for all finite projections $p\in R_\sigma$. So if we take $p_i$ to be any increasing sequence of finite projections such that $p_i\to 1$, then we can make exactly the same argument as above to obtain that $\sigma\precsim\pi$.
\end{proof}


\begin{proof}[Proof of Theorem \ref{mainthmv2} for $\sigma,\pi$ primary.]
We may assume that $\sigma$ and $\pi$ are semifinite primary representations. By the previous lemma we have $\sigma\precsim \pi$, and so if $\sigma$ is properly infinite semifinite (i.e. type $\I_\infty$ or $\II_\infty$), then so is $\pi$, and since it then follows that $\pi$ is isomorphic to its restriction to any infinite projection in $R_\pi$, we have that $\sigma\simeq\pi$. So we may assume that $\sigma$ and $\pi$ are finite. Let $U:H\to H$ and $V:H\to H$ be isometries witnessing that $\sigma\precsim\pi$ and $\pi\precsim\sigma$. Then, since $\sigma$ is not equivalent to its restriction to any proper subprojection of $1$, it follows that $VU=1_H$, and symmetrically, that $UV=1_H$. Thus $U^{-1}=V$, and $U$ witnesses that $\sigma\simeq\pi$.
\end{proof}

We now proceed to consider the case when $\sigma$ and $\pi$ are not necessarily primary representations.


\begin{lemma}\label{l.embedd}
Let $\sigma_i\in\rep(\Gamma,H_i)$, $i\in\{0,1\}$, and let $((X_i,\mu_i),\sigma_{i,x}\actson H_{i,x})$ be the corresponding decompositions into primary representations. Suppose $S_0(\sigma_0,\sigma_1)$ holds. Then there is a measure class preserving injection $\phi:X_0\to X_1$ such that
\begin{equation}\label{eq.decomp}
(\forall^{\mu_0}x\in X_0)S_0(\sigma_{0,x},\sigma_{1,\phi(x)}).
\end{equation}
In  particular, and $\sigma_0\precsim\sigma_1$.
\end{lemma}


\begin{proof}
Without loss of generality, $\mu_0(X_0)=1$.

By Lemma \ref{l.decomp}, for a partial isometry $u\in R_{\sigma_0,\sigma_1}$ we have a measure class preserving injection $\phi_u:\supp(u^*u)\to X_1$ and an assignment $x\mapsto u_x\in R_{\sigma_{0,x},\sigma_{1,\phi(x)}}$ of partial isometries such that 
$$
u=\int \theta(x)^{\frac 1 2} u_x d\mu_0(x).
$$
Note that if $u,v\in R_{\sigma_0,\sigma_1}$ are partial isometries, then $\phi_u(x)=\phi_v(x)$ for almost all $x\in\supp(u)\cap\supp(v)$. To see this, note that $u_x$ witnesses that $\sigma_{0,x}\not\perp\sigma_{1,\phi_u(x)}$ and $v_x$ witnesses that $\sigma_{0,x}\not\perp\sigma_{1,\phi_v(x)}$, and for primary representations the relation $\not\perp$ is transitive. Hence $\phi_u(x)=\phi_v(x)$ a.e. where both are defined, since $\sigma_{1,x}\perp\sigma_{1,x'}$ whenever $x\neq x'$.

Since $S_0(\sigma_0,\sigma_1)$ holds we can find a sequence $(u_n)_{n\in\N}$ in $R^{\rm u}_{\sigma_0,\sigma_1}$ such that $\supp(u_n^*u_n)>1-\frac 1 n$. Defining
$$
\phi=\bigcup_{n\in\N} \phi_{u_n},
$$
it is then clear that $\phi$ defines a measure class preserving injection almost everywhere on $X_0$. We claim that this $\phi$ works.

To see this, let $x\mapsto e_{n,x}\in H_{0,x}$ be a measurable assignment of orthonormal bases\footnote{For simplicity, we assume that $\dim(H_{i,x})=\infty$ a.e.}, and let $e_n=\int e_{n,x}d\mu_0(x)$. By Lemma \ref{l.unitarywitness} it suffices to show that for all $\varepsilon>0$ and $n\in\N$ we have
$$
\alpha_{\varepsilon,n}=\mu_0(\{x\in X_0: (\exists u\in R^{\rm u}_{\sigma_{0,x},\sigma_{1,\phi(x)}})(\forall i\leq n) \|u^*u e_{i,x}-e_{i,x}\|<\varepsilon\})>1-\varepsilon.
$$

For this, let $u\in R^{\rm u}_{\sigma_0,\sigma_1}$ be a partial isometry such that $\|u^*ue_i-e_i\|<\varepsilon^3$ for all $i\leq n$. Decompose $u$ as $u=\int \theta(x)^{\frac 1 2} u_xd\mu_0(x)$ for some assignment $x\mapsto u_x\in R_{\sigma_{0,x},\sigma_{1,\phi(x)}}$. Then since
$$
\|u^*ue_i-e_i\|^2=\int \|u_x^*u_x e_{i,x}-e_{i,x}\|^2 d\mu_0(x)<\varepsilon^3,
$$
it follows that $(1-\alpha_{\varepsilon,n})\varepsilon^2<\varepsilon^3$, and so $\alpha_{\varepsilon,n}>1-\varepsilon$.
\end{proof}


\begin{proof}[Proof of Theorem \ref{mainthmv2}]
Fix direct integral decompositions of $\sigma_0$ and $\sigma_1$ as in Lemma \ref{l.embedd}. Applying Lemma \ref{l.embedd}, there are measure class preserving injections $\phi:X_0\to X_1$ and $\psi: X_1\to X_0$ such that $\sigma_{0,x}\simeq \sigma_{1,\phi(x)}$ and $\sigma_{1,y}\simeq\sigma_{0,\psi(y)}$. Then we must have $\psi=\phi^{-1}$, from which $\sigma_0\simeq\sigma_1$ follows.
\end{proof}


\begin{remark}
Theorem \ref{mainthmv2} is best possible in general, since we have the following:


\begin{prop}\label{prop.complete}
If $\Gamma$ has uncountably many disjoint primary representations then $\simeq\subseteq \rep(\Gamma,H)^2$ is $\mathbf\Pi^0_3$-complete as a set (in the sense of \cite[21.13]{kechris95}.)
\end{prop}

For the proof we need the following: Define as in \cite[2.22]{hjorth00} an equivalence relation on $2^{\N\times\N}$ by
$$
x F_2 y\iff (\forall n)(\exists m)(\forall k) x(n,k)=y(m,k)\wedge (\forall m)(\exists n)(\forall k) x(n,k)=y(m,k).
$$
If we let $x(n,\cdot)\in 2^\N$ denote the $n$'th ``row'' in $x\in 2^{\N\times\N}$, then this means that 
$$
xF_2y\iff \{x(n,\cdot):n\in\N\}=\{y(n,\cdot):n\in\N\}.
$$
Then we have


\begin{lemma}[Folklore]\label{l.F_2}
The relation $F_2$ is a $\mathbf\Pi^0_3$-complete subset of $2^{\N\times\N}\times 2^{\N\times\N}$.
\end{lemma}
\begin{proof}
Recall from \cite[23.A]{kechris95} that the set
$$
P_3=\{x\in 2^{\N\times\N}:(\forall n)(\exists k)(\forall m\geq k) x(n,m)=0\}
$$
is $\mathbf\Pi^0_3$-complete. For $x,y\in 2^{\N\times\N}$, let $x\oplus y\in 2^{\N\times\N}$ be defined by $(x\oplus y)(2k-1,\cdot)=x(k,\cdot)$ and $(x\oplus y)(2k,\cdot)=y(k,\cdot)$, and note that $(x,y)\mapsto x\oplus y$ is continuous. Fix $z\in 2^{\N\times\N}$ such that $z(n,\cdot)$ enumerates all eventually zero sequences in $2^\N$. Then the map
$$
f: 2^{\N\times\N}\to 2^{\N\times\N}\times 2^{\N\times\N}: x\mapsto (x\oplus z,z)
$$
is a continuous reduction of $P_3$ to $F_2$ since $x\in P_3$ iff $(x\oplus z) F_2 z$.
\end{proof}


\begin{proof}[Proof of Proposition \ref{prop.complete}]
It suffices to show that $F_2$ is continuously reducible to $\simeq$ in $\rep(\Gamma,H)$, i.e., there is a continuous $g:2^{\N\times\N}\to\rep(\Gamma,H)$ such that $xF_2y$ if and only if $g(x)\simeq g(y)$.

Let $\vN(H)$ denote the set of von Neumann algebras acting on $H$, and equip this with the Effros Borel structure, as defined in \cite{effros65}. Then it is easy to see that the map $\rep(\Gamma,H)\to\vN(H):\sigma\mapsto R_{\sigma}$ is Borel. It was shown in \cite{effros65} that the set of factors in $\vN(H)$ is a Borel, and so it follows that the set $\mathcal F\subseteq\rep(\Gamma,H)$ of factor representations is Borel. Consider the relation $\not\perp$ in $\mathcal F$. By Remark \ref{rem.transitivity}, $\not\perp$ is a Borel equivalence relation in $\mathcal F$. By assumption, $\not\perp$ has uncountably many classes in $\mathcal F$, and so it follows by Silver's dichotomy theorem (see \cite{silver80} or \cite{gao09}) that there is a continuous $f:2^\N\to\mathcal F$ such that $f(x)\perp f(y)$ whenever $x\neq y$. Further, we can clearly arrange that $f(x)\simeq\bigoplus_{n=1}^\infty f(x)$, the infinite direct sum of countably many copies of $f(x)$. Now define $g:2^{\N\times\N}\to\rep(\Gamma,\bigoplus_{n=1}^\infty H)$ by
$$
g(x)=\bigoplus_{n=1}^\infty f(x(n,\cdot)).
$$
Then $g$ is easily seen to be a continuous reduction of $F_2$ to $\simeq$ in $\rep(\Gamma,\bigoplus_{n=1}^\infty H)$.
\end{proof}

\begin{corollary}
For any countable discrete group $\Gamma$, the conjugacy relation in $\rep(\Gamma,H)$ is $\mathbf\Pi^0_3$-complete.
\end{corollary}

\begin{proof}
If $\Gamma$ is type I, it follows from \cite{thoma68} that it is Abelian by finite, and so it has uncountably many spectrally disjoint unitary representations, and the above argument applies. If $\Gamma$ is not type I, then it has uncountably many non-isomorphic irreducible unitary representations (see \cite{glimm61}) and so Proposition \ref{prop.complete} applies.
\end{proof}

\end{remark}

\begin{remark}
Simon Thomas has pointed out the following interesting consequence of Proposition \ref{prop.complete}:

\begin{corollary}\label{c.simon}
If $\Gamma$ has uncountably many disjoint primary representations then the conjugacy relation in $\rep(\Gamma,H)$ is not Borel reducible to the conjugacy relation on \emph{irreducible} unitary representations of $\Gamma$.
\end{corollary}

\begin{proof}
We refer to \cite{kanovei08} for the relevant notions. The conjugacy relation for irreducible unitary representations is $F_\sigma$ (see e.g. \cite{effros75}), hence it is a pinned equivalence relation by \cite[Theorem 17.1.3]{kanovei08}. However, by that same theorem the equivalence relation $F_2$ is not pinned and not Borel reducible to any pinned equivalence relation.
\end{proof}

That the conclusion of Corollary \ref{c.simon} holds when $\Gamma$ is a discrete countable group of type I is well-known. That this also holds when $\Gamma$ is not type I seems not to have been known.
\end{remark}

\bibliographystyle{amsplain}
\bibliography{descrunitrep}

\bigskip

{\smaller \sc
\noindent  Department of Mathematics and Statistics, University of Melbourne\\
Parkville, 3010 Victoria, Australia
}

\smallskip

{\smaller\sc
\noindent Mathematics Department, University of California, Los Angeles\\
Box 951555, Los Angeles, CA 90095-1555, USA
}

{\smaller \tt \noindent greg@math.ucla.edu}

\bigskip

{\smaller \sc
\noindent  Kurt G\"odel Research Center, University of Vienna\\
W\"ahringer Strasse 25, 1090 Vienna, Austria
}

\smallskip

{\smaller\sc
\noindent Department of Mathematics, University of Copenhagen\\
Universitetsparken 5, 2100 Copenhagen, Denmark
}

{\smaller \tt \noindent asgert@math.ku.dk}

\end{document}